\documentclass{amsart}
\usepackage[utf8]{inputenc}
\usepackage{amsmath, amssymb}
\usepackage{amsthm}
\usepackage[all]{xy}
\usepackage{xcolor}
\definecolor{darkblue}{rgb}{0,0,0.6}
\usepackage[ocgcolorlinks,colorlinks=true, citecolor=darkblue, filecolor=darkblue,
linkcolor=darkblue, urlcolor=darkblue]{hyperref}
\usepackage[nameinlink,capitalise,noabbrev]{cleveref}
\usepackage{bbm}
\usepackage{colonequals}  
\usepackage{enumerate}
\usepackage{tikz-cd}

\makeatletter
\@namedef{subjclassname@2010}{%
  \textup{2010} Mathematics Subject Classification}
\makeatother
\title{$K_1$-groups via binary complexes of fixed length}
\author[D.~Kasprowski]{Daniel Kasprowski}
\address{Rheinische Friedrich-Wilhelms-Universit\"at Bonn, Mathematisches Institut,\newline\indent Endenicher Allee 60, 53115 Bonn, Germany}
\email{kasprowski@uni-bonn.de}
\urladdr{http://www.math.uni-bonn.de/people/daniel/}

\author[B.~K\"ock]{Bernhard K\"ock}
\address{School of Mathematical Sciences, University of Southampton, Highfield, \newline \indent Southampton SO17 1BJ, United Kingdom}
\email{b.koeck@soton.ac.uk}
\urladdr{https://www.southampton.ac.uk/maths/about/staff/bk2.page}

\author[C.~Winges]{Christoph Winges}
\address{Rheinische Friedrich-Wilhelms-Universit\"at Bonn, Mathematisches Institut,\newline\indent Endenicher Allee 60, 53115 Bonn, Germany}
\email{winges@math.uni-bonn.de}
\urladdr{http://www.math.uni-bonn.de/people/winges/}


\newcounter{commentcounter}

\date{\today}

\newcommand{\bbC}{\mathbbm{C}}
\newcommand{\bbQ}{\mathbb{Q}}

\newcommand{\bbP}{\mathbb{P}}

\DeclareMathOperator{\id}{id}

\newcommand{\cN}{\mathcal{N}}

\newcommand{\bD}{\mathbb{D}}
\newcommand{\bJ}{\mathbb{J}}
\newcommand{\bP}{\mathbb{P}}
\newcommand{\bQ}{\mathbb{Q}}

\newcommand{\bT}{\mathbb{T}}

\newcommand{\comment}[1]{}
\newcommand{\bM}{\mathbb{M}}
\newcommand{\bN}{\mathbb{N}}

\renewcommand{\phi}{\varphi}

\numberwithin{equation}{section}

\newtheorem{theorem}[equation]{Theorem}
\newtheorem{prop}[equation]{Proposition}
\newtheorem{cor}[equation]{Corollary}
\newtheorem{lemma}[equation]{Lemma}
\theoremstyle{definition}

\newtheorem{defi}[equation]{Definition}
\newtheorem{definition}[equation]{Definition}
\newtheorem{rem}[equation]{Remark}
\newtheorem{remark}[equation]{Remark}

\newcommand{\lad}{i}
\DeclareMathOperator{\short}{\mathrm{sh}}

\newcommand{\tors}[2]{\langle #1,#2 \rangle}

\begin{document}

\begin{abstract}
 We modify Grayson's model of $K_1$ of an exact category to give a presentation whose generators are binary acyclic complexes of length at most~$k$ for any given $k \ge 2$.
 As a corollary, we obtain another, very short proof of the identification of Nenashev's and Grayson's presentations.
\end{abstract}

\maketitle

\section{Introduction}
Let $\cN$ be an exact category. Algebraic descriptions of Quillen's $K_1$-group of $\cN$ in terms of explicit generators and relations have been given by Nenashev \cite{Nenashev1998} and Grayson \cite{Grayson2012}. The generators in both descriptions are so-called binary acyclic complexes in $\cN$. While Nenashev uses complexes of length at most $2$, Grayson's generators are of arbitrary (finite) length. An algebraic proof of the fact that these two descriptions agree has been given in \cite{KW17}. In this paper, we will  give another presentation of $K_1(\cN)$; this time the generators are binary acyclic complexes of length at most~$k$ for any $k \ge 2$, see \cref{def:L1k} and \cref{thm:short-complexes}. A motivating question behind this new description is to determine the precise relations that in addition to Grayson's relations need to be divided out when restricting the generators in Grayson's description to complexes of length at most~$k$. If $k=2$, our relations are special cases of Nenashev's relations. In this sense, our presentation simplifies Nenashev's presentation. All this leads to a new, natural and sleek algebraic proof of the fact that Nenashev's and Grayson's descriptions agree, see \cref{sec:nena}.

The proof of our main result, see \cref{sec:proof}, basically proceeds by induction on $k$. The crucial ingredient in the inductive step is a shortening procedure for binary acyclic complexes, see \cref{def:shortening}, as discovered by Grayson and also used in~\cite{KW17}. The main new idea in this paper is the comparatively short and simple way of showing how this shortening procedure yields an inverse to passing from complexes of length $k$ to complexes of length $k+1$, see \cref{prop:shorten}.

\subsection*{Acknowledgements}
We thank the referee for their quick and careful report and for a question, which we answer in \cref{rem:NenashevRelation}.

C.~Winges acknowledges support by the Max Planck Society and Wolfgang L\"uck's ERC Advanced Grant ``KL2MG-interactions" (no.~662400). D.~Kasprowski and C.~Winges were funded by the Deutsche Forschungsgemeinschaft (DFG, German Research Foundation)
under Germany's Excellence Strategy - GZ 2047/1, Projekt-ID 390685813.

\section{Background and Statement of Main Theorem}

We recall a {\em binary acyclic complex} $\bbP=(P_*, d, d')$ in~$\cN$ is a graded object $P_*$ in $\cN$ supported on a finite subset of $[0,\infty]$ together with two degree $-1$ maps $d, d':P_* \rightarrow P_*$ such that both $(P, d)$ and $(P_*,d')$ are acyclic chain complexes in $\cN$. Here, {\em acyclic} means that each differential $d_n: P_n \rightarrow P_{n-1}$ admits a factorisation into an admissible epimorphism followed by an admissible monomorphism
\[P_n \twoheadrightarrow J_{n-1} \rightarrowtail P_{n-1}\]
such that $J_n \rightarrowtail P_n \twoheadrightarrow J_{n-1}$ is a short exact sequence in $\cN$ for every $n$. The differentials $d$ and $d'$ are called the {\em top} and {\em bottom} differential. We also write $\bbP^\top$ and $\bbP^\bot$ for the complexes $(P_*, d)$ and $(P_*, d')$. If $d=d'$,  we call $\bbP$ a {\em diagonal} binary acylic complex.

A {\em morphism between binary acyclic complexes} $\bbP$ and $\bbQ$ is a degree 0 map between the underlying graded objects which is a chain map with respect to both differentials. According to \cite[Section~3]{Grayson2012}, the obvious definition of short exact sequences turns the category of binary acyclic complexes into an exact category. We denote its Grothendieck group by $B_1(\cN)$.

\begin{definition}\label{def:GraysonK1}
The quotient of $B_1(\cN)$ obtained by declaring the classes of diagonal binary acyclic complexes to be zero is called {\em Grayson's $K_1$-group of $\cN$} and denoted by $K_1(\cN)$.
\end{definition}

Grayson proves in \cite[Corollary 7.2]{Grayson2012} that $K_1(\cN)$ is naturally isomorphic to Quillen's $K_1$-group of $\cN$. This justifies our notation. Note that Grayson uses complexes supported on $[-\infty, \infty]$. By \cite[Proposition~1.4]{HKT2017}, defining $K_1(\cN)$ with complexes supported on~$[0,\infty]$ as above yields the same $K_1$-group.

If we restrict the generators to be complexes supported on~$[0,k]$ (for $k \ge 0$), we write $B^k_1(\cN)$ and $K^k_1(\cN)$ for the resulting abelian groups.

\begin{definition}\label{def:L1k}\mbox{}
 \begin{enumerate}[(a)]
  \item\label{it:L1ka} A \emph{binary ladder} is a quadruple $(\bP,\bQ,\sigma,\tau)$ consisting of two binary complexes $\bP$ and $\bQ$ together with isomorphisms $\sigma \colon \bP^\top \xrightarrow{\sim} \bQ^\top$ and $\tau \colon \bP^\bot \xrightarrow{\sim} \bQ^\bot$.
\item\label{it:L1kc} Any two isomorphisms $\alpha, \beta \colon P \xrightarrow{\sim} Q$ in $\cN$ define a binary acyclic complex
 \[\begin{tikzcd}
  P\ar[r, shift left, "\alpha"]\ar[r, shift right, "\beta"'] & Q.
  \end{tikzcd}\]
 in $\cN$ supported on $[0,1]$. The corresponding element in any of the groups $B_1(\cN)$, $K_1(\cN)$ or $B_1^k(\cN)$, $K_1^k(\cN)$ for $k \ge 1$ is denoted by $\tors{\alpha}{\beta}$.
 \item\label{it:L1kd} Define $L^k_1(\cN)$ to be the quotient of $K_1^k(\cN)$ obtained by additionally imposing the relation
 \[ \bQ - \bP = \sum_{i=0}^k (-1)^i \tors{\sigma_i}{\tau_i} \]
 for every binary ladder $(\bP,\bQ,\sigma,\tau)$ in $\cN$ such that $\bP$ and $\bQ$ are supported on $[0,k]$, $P_i=Q_i$ and all $\sigma_i$ and $\tau_i$ are involutions, i.e.\ $\sigma_i^2=\id=\tau_i^2$.
 \end{enumerate}
\end{definition}

\begin{remark}
The object $L_1^k(\cN)$ has of course nothing to do with $L$-Theory; the~$L$ here is rather meant to refer to `ladder'.
It will follow from \cref{prop:shorten}, that imposing the ladder relation in \cref{def:L1k}\eqref{it:L1kd} not just for $\sigma_i,\tau_i$ involutions but for all automorphisms yields an isomorphic group.
\end{remark}

The following theorem is the precise formulation of our main result. Note that \cref{lem:factorisation} below implies that we have a natural map $L_1^k(\cN) \rightarrow K_1(\cN)$.

\begin{theorem}\label{thm:short-complexes}
 The canonical map
 \[ L_1^k(\cN) \to K_1(\cN) \]
 is an isomorphism for every $k \geq 2$.
\end{theorem}

\cref{sec:proof} contains the proof of \cref{thm:short-complexes}.

\section{Proof of the main theorem}\label{sec:proof}
\begin{defi}
	\label{it:L1kb} For any object $P$ in $\cN$ let
	\[ \tau_P := \begin{pmatrix} 0 & \id \\ \id & 0 \end{pmatrix} \colon P \oplus P \xrightarrow{\sim} P \oplus P \]
	denote the automorphism of $P \oplus P$ which switches the two summands.
\end{defi}

Note that assigning $\tors{\id_{P\oplus P}}{\tau_P}$ with any object $P \in \cN$ defines a homomorphism $K_0(\cN) \rightarrow B_1^1(\cN)$. In particular, $\tors{\id_{P\oplus P}}{\tau_P} = \tors{\id_{Q\oplus Q}}{\tau_Q}$ if $P$ and $Q$ represent the same element in $K_0$.
As an aside, we remark that $\tors{\id_{P\oplus P}}{\tau_P}$ is equal to $\tors{\id_P}{-\id_P}$ in $B_1^1(\cN)$ and of order at most two in $L_1^1(\cN)$; both of these two facts are easy to prove but won't be used in this paper.

\begin{lemma}\label{lem:switching} Let $\bP$ be a binary acyclic complex in $\cN$ supported on $[0,k]$ and let $\mathrm{sw}(\bP)$ denote the binary complex obtained from $\bP$ by switching top and bottom differential. Then we have
\[\mathrm{sw}(\bP) = - \bP \qquad \textrm{ in } \qquad L_1^k(\cN).\]
\end{lemma}

\begin{proof} We have $\bP + \mathrm{sw}(\bP) = \bP \oplus \mathrm{sw}(\bP)  \textrm{ in }  B_1^k(\cN)$. The latter complex represents~$0$ in $L_1^k(\cN)$. To see this, consider the binary ladder $(\bP \oplus \mathrm{sw}(\bP), \bD, \sigma, \tau)$ where $\bD$ is the diagonal complex with $\bD^\top = \bD^\bot = (\bP \oplus \mathrm{sw}(\bP))^\top$, $\sigma = \id$ and $\tau$ switches the two summands $\bP$ and $\mathrm{sw}(\bP)$, and note that $\sum_{i=0}^k (-1)^i P_i = 0$ in $K_0(\cN)$.
\end{proof}

Regarding binary acyclic complexes supported on $[0,k]$ as complexes supported on $[0,k+1]$ defines a homomorphism
 \[ i_k \colon L^k_1(\cN) \to L^{k+1}_1(\cN). \]

The following lemma is basically a special case of the generalised Nenashev relation, see \cref{def:NenashevK1} below and \cite[Proposition~2.12]{HarPhD}. We include a short proof to convince the reader that complexes supported on $[0,k+1]$ suffice to prove the desired relation.
As usual, we write $\bP[1]$ for the complex shifted by $1$ (without changing the sign of the differentials $d$ and $d'$) so that $\bP[1]_0 =0$.

\begin{lemma}\label{lem:factorisation}
The homomorphism $i_k$ naturally factorises as
 \[ i_k \colon L^k_1(\cN) \to K^{k+1}_1(\cN) \twoheadrightarrow L^{k+1}_1 (\cN) \]
 where the second map is the canonical epimorphism.
\end{lemma}

\begin{proof} Let $(\bP,\bQ,\sigma,\tau)$ be a binary ladder with $\bP, \bQ$ supported on $[0,k]$.
 Then all rows and columns of the diagram
 \[\begin{tikzcd}
  \ldots\ar[r, shift right]\ar[r, shift left] & P_2\ar[d, shift right, "\sigma_2"']\ar[d, shift left, "\tau_2"]\ar[r, shift left]\ar[r, shift right] & P_1\ar[d, shift right, "\sigma_1"']\ar[d, shift left, "\tau_1"]\ar[r, shift left]\ar[r, shift right]& P_0\ar[d, shift right, "\sigma_0"']\ar[d, shift left, "\tau_0"] \\
 \ldots\ar[r, shift right]\ar[r, shift left] & Q_2\ar[r, shift left]\ar[r, shift right] & Q_1\ar[r, shift left]\ar[r, shift right]& Q_0
 \end{tikzcd}\]
 are binary acyclic complexes, top differentials commute with top differentials and bottom differentials commute with bottom differentials.
 Filtering the associated total complex $\bT$ (which is a binary acyclic complex supported on $[0,k+1]$) ``horizontally and vertically'' then yields the relation
 \[ \bQ + \bP[1] = \bT = \sum_{i=0}^k  \tors{\sigma_i}{\tau_i}[i] \]
 in $B^{k+1}_1(\cN)$. If $\bP = \bQ$ and $\sigma=\tau = \id$, this shows that
 \begin{equation}\label{equ:shifting} \bP[1] = - \bP
 \end{equation}
 in $K_1^{k+1}(\cN)$. The two equalities above finally show that
  \[ \bQ -\bP = \sum_{i=0}^k  (-1)^i\tors{\sigma_i}{\tau_i}\]
 in $K^{k+1}_1(\cN)$, as desired.
\end{proof}

\begin{definition}\label{def:shortening}
 Let $k \ge 2$ and let $\bP = (P_*,d,d')$ be a binary acyclic complex supported on $[0,k+1]$,
 and choose factorisations
 \[ d_2 \colon P_2 \twoheadrightarrow J \rightarrowtail P_1 \quad \text{and} \quad
  d_2' \colon P_2 \twoheadrightarrow K \rightarrowtail P_1 \]
  witnessing that $(P_*,d)$ and $(P_*,d')$ are acyclic. In the following, we denote the maps $P_2\to J$ and $P_2\to K$ again by $d_2$ and $d_2'$.

 The {\em Grayson shortening of $\bP$} is the  binary acyclic complex $\short(\bP)$ supported on~$[0,k]$ whose top component is given by
 \[\begin{tikzcd}[row sep= tiny]
 \ldots \ar[r] & P_3\ar[r, "d_3"]& P_2\ar[r, "d_2"]\ar[d, phantom, "\oplus"] & J\ar[d, phantom, "\oplus"] \\
 & \oplus & K\ar[r, "\id"]\ar[d, phantom, "\oplus"] & K \\
 & J\ar[r, "\id"]\ar[d, phantom, "\oplus"] & J\ar[d, phantom, "\oplus"] & \oplus \\
 & K\ar[r, rightarrowtail] & P_1\ar[r, "d_1'"] & P_0 \\
 \end{tikzcd}\]
 and whose bottom component is
  \[\begin{tikzcd}[row sep= tiny]
 \ldots \ar[r] & P_3\ar[r, "d_3'"] & P_2\ar[r, "d_2'"]\ar[d, phantom, "\oplus"] & K\ar[d, phantom, "\oplus"] \\
 & \oplus & J\ar[r, "\id"]\ar[d, phantom, "\oplus"] & J \\
 & K\ar[r, "\id"]\ar[d, phantom, "\oplus"] & K\ar[d, phantom, "\oplus"] & \oplus \\
 & J\ar[r, rightarrowtail] & P_1\ar[r, "d_1"] & P_0 \\
 \end{tikzcd}\]
 Note that we have permuted the summands in the bottom component for better legibility, but that we consider the summation order in the top component to be the definitive one.
\end{definition}
\begin{rem}
	\label{rem:length1}
	Note that if $\bP$ is supported on $[0,1]$ then $\short(\bP)=\mathrm{sw}(\bP)$ and thus $\short(\bP)=-\bP\in L_1^1(\cN)$ in this case by \cref{lem:switching}.
\end{rem}
\begin{rem} \label{rem:GraysonShortening} The complex $\short(\bP)$ appears in handwritten notes by Grayson and has also been used in \cite[Section~5]{KW17}.
Note that our definition of $\short(\bP)$ includes a shift by $-1$ so $\short(\bP)$ is supported on $[0,k]$ rather than on $[1,k+1]$. This avoids bulky notations later.
\end{rem}

For $\bP$, $J$ and $K$ as in \cref{def:shortening}, we have $\tors{\id}{\tau_J} = \tors {\id}{\tau_K}$ because $J=K$ in $K_0(\cN)$. We denote the latter element by $\tau_\bP$. If in fact $J \cong K$, we replace the morphisms $P_2 \twoheadrightarrow K$ and $K \rightarrowtail P_1$ with $P_2 \twoheadrightarrow J$ and $J \rightarrowtail P_1$ by composing them with a fixed isomorphism between $J$ and $K$. Then the ordinary (non-naive) truncations
\[ \mathrm{t}_{\ge 1}(\bP) := (\begin{tikzcd}\ldots\ar[r, shift right]\ar[r, shift left] & P_3\ar[r, shift left]\ar[r, shift right] & P_2\ar[r, shift left]\ar[r, shift right]& J \end{tikzcd} )\]
and
\[ \mathrm{t}_{\le 2}(\bP) := (\begin{tikzcd}J\ar[r, shift left]\ar[r, shift right] & P_1\ar[r, shift left]\ar[r, shift right]& P_0\end{tikzcd} )\]
are binary acyclic complexes again. In this case, the following crucial lemma computes $\short(\bP) \in L_1^k(\cN)$ in terms of these truncations and $\tau_\bP$.

\begin{lemma}\label{lem:truncation} Let $\bP$ be a binary acyclic complex supported on $[0,k+1]$ and suppose that $J \cong K$. Then we have
\[\short(\bP) = \mathrm{t}_{\ge 1}(\bP)[-1] - \mathrm{t}_{\le 2}(\bP) - \tau_\bP \qquad \textrm{in} \qquad L_1^k(\cN).\]
In particular:
\begin{enumerate}[(a)]
 \item\label{it:truncationa} If $\bP$ is a diagonal complex, then $\short(\bP) = - \tau_\bP$ in $L_1^k(\cN)$.
 \item\label{it:truncationb} If $P_0 = 0$, then $\short(\bP) = \bP[-1] - \tau_\bP$ in $L_1^k(\cN)$.
 \end{enumerate}
\end{lemma}

\begin{proof}
Without permuting any summands in the bottom component, $\short(\bP)$ has top component
 \[\begin{tikzcd}[row sep= tiny]
 \ldots \ar[r] & P_3\ar[r, "d_3"] & P_2\ar[r, "d_2"]\ar[d, phantom, "\oplus"] & J\ar[d, phantom, "\oplus"] \\
 & \oplus & J\ar[r, "\id"]\ar[d, phantom, "\oplus"] & J\\
 & J\ar[r, "\id"]\ar[d, phantom, "\oplus"] & J\ar[d, phantom, "\oplus"] & \oplus \\
 & J\ar[r, rightarrowtail] & P_1\ar[r, "d'_1"] & P_0 \\
 \end{tikzcd}\]
 and bottom component
 \[\begin{tikzcd}[row sep= tiny]
 \ldots \ar[r] & P_3\ar[r, "d'_3"]& P_2\ar[dr, near start, "d'_2"]\ar[d, phantom, "\oplus"] & J\ar[d, phantom, "\oplus"] \\
 & \oplus & J\ar[d, phantom, "\oplus"] & J \\
 & J\ar[dr, rightarrowtail]\ar[d, phantom, "\oplus"] & J\ar[d, phantom, "\oplus"] \ar[uur, near start, "\id"'] & \oplus \\
 & J\ar[uur, "\id"] & P_1\ar[r, "d_1"] & P_0 \\
 \end{tikzcd}\]
  Now consider the binary ladder $(\short(\bP), \mathrm{t}_{\ge 1}(\bP)[-1] \oplus \bJ \oplus \bJ[1] \oplus \mathrm{sw}(\mathrm{t}_{\le 2}(\bP)), \sigma, \tau)$, where $\bJ$ is the diagonal binary complex supported on $[0,1]$ given by $\bJ^\top = \bJ^\bot = (J \overset{\id}{\longrightarrow} J)$, $\sigma = \id$  and $\tau$ is the automorphism switching the two copies of $J$ in degrees $0$, $1$ and~$2$ and the identity in all higher degrees. From this binary ladder we obtain the following equality in $L_1^k(\cN)$:
 \[  \mathrm{t}_{\ge 1}(\bP)[-1] +  \mathrm{sw}(\mathrm{t}_{\le 2}(\bP)) = \short(\bP) + \tau_\bP .\]
Using \cref{lem:switching}, we finally obtain the desired equality in~$L_1^k(\cN)$:
  \[\short(\bP) = \mathrm{t}_{\ge 1}(\bP)[-1] - \mathrm{t}_{\le 2}(\bP) - \tau_\bP. \]

If $\bP$ is a diagonal complex, both truncations are diagonal again and part~\eqref{it:truncationa} follows. If $P_0 =0$, then $\mathrm{t}_{\ge 1}(\bP) = \bP$, $J=P_1$ and $\mathrm{t}_{\le 2}(\bP) = \tors{\id_{P_1}}{\id_{P_1}}[1]=0$ in $L_1^2(\cN)$; this shows part~\eqref{it:truncationb}.
\end{proof}

\begin{prop}\label{prop:shorten}
 The homomorphism
 \[ i_k \colon L_1^k(\cN) \to L_1^{k+1}(\cN) \]
 is an isomorphism for all $k \geq 2$.
 Its inverse is induced by the assignment
 \[ \bP \mapsto - \short(\bP) - \tau_\bP. \]
\end{prop}
\begin{proof}
 We begin by showing that the assignment $\bP \mapsto - \short(\bP)- \tau_\bP$ induces a well-defined homomorphism $p_k: L_1^{k+1}(\cN) \to L^k_1(\cN)$.

 If $\bP' \rightarrowtail \bP \twoheadrightarrow \bP''$ is a short exact sequence of binary acyclic complexes supported on $[0,k+1]$, then  we have induced short exact sequences $J' \rightarrowtail J \twoheadrightarrow J''$ and $K' \rightarrowtail K \twoheadrightarrow K''$ by \cite[Corollary~3.6]{Bu10}. In particular, we obtain a short exact sequence $\short(\bP') \rightarrowtail \short(\bP) \twoheadrightarrow \short(\bP'')$. Since we also have $\tau_\bP = \tau_{\bP'} + \tau_{\bP''}$, we obtain an induced homomorphism
 \[ p_k'' \colon B_1^{k+1}(\cN) \to L^k_1(\cN). \]
 Let $\bP$ be a diagonal binary acyclic complex supported on $[0,k+1]$, then we have $p''_k(\bP) = 0$ in $L_1^k(\cN)$ by \cref{lem:truncation}\eqref{it:truncationa}. In particular, $p_k''$ induces a homomorphism
 \[ p_k' \colon K_1^{k+1}(\cN) \to L_1^k(\cN). \]
 If $(\bP,\bQ,\sigma,\tau)$ is a binary ladder in $\cN$ with $\bP$ and $\bQ$ supported on $[0,k+1]$, then $\sigma_1$ and $\tau_1$ induce isomorphisms $\sigma^J \colon J_\bP \xrightarrow{\sim} J_\bQ$ and $\tau^K \colon K_\bP \xrightarrow{\sim} K_\bQ$, respectively. If $\sigma_1$ and $\tau_2$ are involutions, so are $\sigma^J$ and $\tau^K$.
 These define a binary ladder $(\short(\bP), \short(\bQ), \short(\sigma), \short(\tau))$ with
 \begin{align*}
  &\short(\sigma)_0 = \sigma^J \oplus \tau^K \oplus \tau_0, \quad
  \short(\sigma)_1 = \sigma_2 \oplus \tau^K \oplus \sigma^J \oplus \tau_1, \quad
  \short(\sigma)_2 = \sigma_3 \oplus \sigma^J \oplus \tau^K \\
  &\short(\tau)_0 = \sigma^J \oplus \tau^K \oplus \sigma_0, \quad
  \short(\tau)_1 = \tau_2 \oplus \tau^K \oplus \sigma^J \oplus \sigma_1, \quad
  \short(\tau)_2 = \tau_3 \oplus \sigma^J \oplus \tau^K
  \end{align*}
  and $\short(\sigma)_i = \sigma_{i+1}$ and $\short(\tau)_i = \tau_{i+1}$ for $i \geq 3$.
  Hence we have
  \begin{eqnarray*}
  \lefteqn{\sum_{i=0}^{k} (-1)^i \tors{\short(\sigma)_i}{\short(\tau)_i}}\\
   &= &\tors{\tau_0}{\sigma_0} - \tors{\sigma_2}{\tau_2} - \tors{\tau_1}{\sigma_1} + \tors{\sigma_3}{\tau_3} + \sum_{i=3}^\infty (-1)^i\tors{\sigma_{i+1}}{\tau_{i+1}} \\
   &= & - \sum_{i=0}^{k+1} (-1)^{i} \tors{\sigma_i}{\tau_i}
   \end{eqnarray*}
in $L_1^1(\cN)$ by \cref{lem:switching}.
  If the given ladder is as in \cref{def:L1k}\eqref{it:L1kd}, it follows  that
  \begin{eqnarray*}
  \lefteqn{p_k'(\bQ - \bP) =  -\short(\bQ) - \tau_\bQ + \short(\bP) + \tau_\bP}\\
  &=& - \sum_{i\ge 0} (-1)^i \tors{\short(\sigma)_i}{\short(\tau)_i}
  =  p'_k\left(\sum_{i \ge 0} (-1)^i\tors{\sigma_i}{\tau_i}\right)
  \end{eqnarray*}
  in $L_1^k(\cN)$   since $J_\bP \cong J_\bQ$, hence $\tau_\bP = \tau_\bQ$, and since, by \cref{rem:length1},  $p'_k$ maps  $\tors{\sigma_i}{\tau_i}\in K_1^{k+1}(\cN)$ to $\tors{\sigma_i}{\tau_i}\in L_1^{k}(\cN)$ for each $i$. Consequently, $p_k'$ induces the desired homomorphism
  \[ p_k \colon L_1^{k+1}(\cN) \to L_1^k(\cN). \]

  We now show that $p_{k} \circ \lad_k = \id$ for all $k \geq 2$. Let $\bP$ be a binary acyclic complex in $\cN$ supported on $[0,k]$.
  By \cref{equ:shifting} and \cref{lem:truncation}\eqref{it:truncationb}, we have
  \[p_k (i_k(\bP)) = - p_k( \bP[1]) = \short(\bP[1]) + \tau_{\bP[1]} = (\bP - \tau_{\bP[1]}) + \tau_{\bP[1]} = \bP. \]
  in $L_1^k(\cN)$, as was to be shown.

  We are left with showing that $p_k$ is also a right-inverse to $i_k$.
  Let $\bP$ be a binary acyclic complex in $\cN$ supported on $[0,k+1]$.
  Since the definition of $\short(\bP)$ is independent of whether we regard $\bP$ as a complex supported on $[0,k+1]$ or $[0,k+2]$, we have
  \[\lad_k ( p_{k}(\bP)) = p_{k+1} ( \lad_{k+1} (\bP))= \bP\]
in $L_1^{k+1}(\cN)$, as desired.
\end{proof}
\begin{proof}[Proof of \cref{thm:short-complexes}]
 From \cref{lem:factorisation} we obtain the directed system
 \[ K_1^2(\cN) \twoheadrightarrow L_1^2(\cN) \to K_1^3(\cN) \twoheadrightarrow L_1^3(\cN) \rightarrow \ldots \]
 Since the colimit of the cofinal sub-system  $K_1^2(\cN) \rightarrow K_1^3(\cN) \rightarrow \ldots$ is $K_1(\cN)$, the colimit of the displayed system is $K_1(\cN)$ as well. Hence, the colimit of the cofinal sub-system $L_1^2(\cN) \rightarrow L_1^3(\cN) \to \ldots$ is also $K_1(\cN)$. Furthermore, all the connecting maps in this sub-system are isomorphisms by \cref{prop:shorten}. The claim follows.
\end{proof}

\begin{remark}
Grayson shows in the handwritten notes mentioned in \cref{rem:GraysonShortening} that $\short(\bP) \in K_1(\cN)$ differs from $\bP$ by classes of binary acyclic complexes of length at most~$2$. By induction, this proves that the canonical map $K_1^2(\cN) \rightarrow K_1(\cN)$ is surjective. While Grayson uses slightly involved double complex arguments, we use simpler and at the same time more potent arguments and also prove the simple relation $\bP + \short(\bP) = - \tau_\bP$ in $K_1(\cN)$.
\end{remark}

\begin{cor}[{\cite[Theorem~1.4]{KW17}}]
	The canonical map $K_1^3(\cN)\to K_1(\cN)$ is onto and admits a canonical section.
\end{cor}
\begin{proof}
	The right inverse is given by the inverse of the bijection $L_1^2(\cN)\to K_1(\cN)$ from \cref{thm:short-complexes} composed with the map $L_1^2(\cN)\to K_1^3(\cN)$ from \cref{lem:factorisation}.
\end{proof}
\begin{rem}
	The inverse of the isomorphism $L_1^2(\cN)\to K_1(\cN)$ from \cref{thm:short-complexes} admits an explicit description. This agrees with the map $\Psi$ appearing in the proof of \cite[Theorem~1.1]{KW17}.
\end{rem}

\section{The relation to Nenashev's description}\label{sec:nena}

In this section, we compare Nenashev's and Grayson's descriptions of $K_1$.

\begin{definition}\label{def:NenashevK1}
{\em Nenashev's $K_1$-group $K_1^\mathrm{N}(\cN)$ of $\cN$} is defined as the abelian group generated by binary acylic complexes $\bP$ of length $2$ subject to the following relations:
\begin{enumerate}[(1)]
 \item\label{it:NenashevK1a} If $\bP$ is a diagonal complex, then $\bP=0$.
 \item\label{it:NenashevK1b} If
 \[\begin{tikzcd}
   P'_2\ar[d, shift right]\ar[d, shift left]\ar[r, shift left]\ar[r, shift right] & P'_1\ar[d, shift right]\ar[d, shift left]\ar[r, shift left]\ar[r, shift right]& P'_0\ar[d, shift right]\ar[d, shift left] \\
  P_2\ar[d, shift right]\ar[d, shift left]\ar[r, shift left]\ar[r, shift right] & P_1\ar[d, shift right]\ar[d, shift left]\ar[r, shift left]\ar[r, shift right]& P_0\ar[d, shift right]\ar[d, shift left]\\
  P''_2\ar[r, shift left]\ar[r, shift right] & P''_1\ar[r, shift left]\ar[r, shift right]& P''_0
 \end{tikzcd}\]
is a diagram in $\cN$ such that all rows and columns are binary acyclic complexes, top differentials commute with top differentials and bottom differentials commute with bottom differentials, then
\[\bP_0 - \bP_1 +\bP_2 = \bP' - \bP + \bP''.\]
 \end{enumerate}
\end{definition}

Nenashev proves in \cite{Nenashev1998} that $K_1^\mathrm{N}(\cN)$ is canonically isomorphic  to Quillen's $K_1$-group of $\cN$. The following corollary purely algebraically proves that $K_1^\mathrm{N}(\cN)$ is isomorphic to $K_1(\cN)$, i.e., to Grayson's $K_1$-group of $\cN$. By \cite[Remark~8.1]{Grayson2012}, regarding a binary acyclic complex of length $2$ as a class in $K_1(\cN)$ defines a map $K_1^\mathrm{N}(\cN) \rightarrow K_1(\cN)$.

\begin{cor}[{\cite[Theorem~1.1]{KW17}}]\label{cor:nenashev}
 The canonical map
 \[ K_1^\mathrm{N}(\cN) \to K_1(\cN) \]
 is an isomorphism.
\end{cor}
\begin{proof}
 Since the relations used to define $L_1^2(\cN)$ are special cases of Nenashev's relation, the canonical surjection $K_1^2(\cN) \twoheadrightarrow K_1^\mathrm{N}(\cN)$ factors via $L_1^2(\cN)$, yielding a surjection $L_1^2(\cN) \twoheadrightarrow K_1^\mathrm{N}(\cN)$.
 Since we have a commutative diagram
 \[\begin{tikzcd}
  L_1^2(\cN)\ar[r]\ar[d] & K_1(\cN) \\
  K_1^\mathrm{N}(\cN)\ar[ur] &
 \end{tikzcd}\]
 it follows from \cref{thm:short-complexes} that
  $L_1^2(\cN) \to K_1^\mathrm{N}(\cN)$ and $K_1^\mathrm{N}(\cN) \rightarrow K_1(\cN)$ are isomorphisms.
\end{proof}

\begin{rem}\label{rem:NenashevRelation}
The bijectivity of the map $L_1^2(\cN) \rightarrow K_1^\mathrm{N}(\cN)$ in the proof of the previous corollary means that Nenashev's relation (\cref{def:NenashevK1}\eqref{it:NenashevK1b}) can be expressed in $K_1^2(\cN)$ as a linear combination of relations arising from binary ladders that we used to define $L_1^2(\cN)$ (\cref{def:L1k}\eqref{it:L1kd}). The object of this remark is to explicitly write down such a linear combination.

Let
\[\begin{tikzcd}
M_2\ar[d, shift right]\ar[d, shift left]\ar[r, shift left]\ar[r, shift right] & M_1\ar[d, shift right]\ar[d, shift left]\ar[r, shift left]\ar[r, shift right]& M_0\ar[d, shift right]\ar[d, shift left] \\
N_2\ar[d, shift right]\ar[d, shift left]\ar[r, shift left]\ar[r, shift right] & N_1\ar[d, shift right]\ar[d, shift left]\ar[r, shift left]\ar[r, shift right]& N_0\ar[d, shift right]\ar[d, shift left]\\
P_2\ar[r, shift left]\ar[r, shift right] & P_1\ar[r, shift left]\ar[r, shift right]& P_0
\end{tikzcd}\]
be a diagram in $\cN$ as in \cref{def:NenashevK1}\eqref{it:NenashevK1b}.
Let $\bT$ denote the associated binary total complex (of length $4$). Choose factorisations
\[ M_1 \oplus N_2 \twoheadrightarrow J_3 \rightarrowtail M_0 \oplus N_1 \oplus P_2 \twoheadrightarrow J_2 \rightarrowtail N_0 \oplus P_1 \]
of the second and third top differential of $\bT$, and define $K_3$ and $K_2$ analogously in terms of the bottom differentials of $\bT$. The next step, according to our earlier constructions, would be to apply the Grayson shortening twice to $\bT$, but this results in a complicated complex with superfluous terms. We rather apply twice just the idea behind the Grayson shortening in order to obtain the following complex $\bT'$ of length $2$ with obvious differentials (top differentials on the left hand side, bottom differentials on the right hand side):
\[\begin{tikzcd}[row sep= tiny]
M_2\ar[r] & M_1 \oplus N_2\ar[r]\ar[d, phantom, "\oplus"] & J_3\ar[d, phantom, "\oplus"] \\
\oplus & K_3\ar[r]\ar[d, phantom, "\oplus"] & K_3\\
J_3\ar[r]\ar[d, phantom, "\oplus"] & J_3\ar[d, phantom, "\oplus"] & \oplus \\
K_3\ar[r] & M_0 \oplus N_1 \oplus P_2\ar[r]\ar[d, phantom, "\oplus"] & K_2\ar[d, phantom, "\oplus"] \\
\oplus & J_2\ar[r]\ar[d, phantom, "\oplus"] & J_2 \\
K_2\ar[r]\ar[d, phantom, "\oplus"] & K_2\ar[d, phantom, "\oplus"] & \oplus \\
J_2\ar[r] & N_0 \oplus P_1\ar[r] & P_0 \\
\end{tikzcd}
\qquad
\begin{tikzcd}[row sep= tiny]
M_2\ar[r] & M_1 \oplus N_2\ar[rd]\ar[d, phantom, "\oplus"] & J_3\ar[d, phantom, "\oplus"] \\
\oplus & K_3\ar[d, phantom, "\oplus"] & K_3\\
J_3\ar[rd]\ar[d, phantom, "\oplus"] & J_3\ar[uur]\ar[d, phantom, "\oplus"] & \oplus \\
K_3\ar[uur] & M_0 \oplus N_1 \oplus P_2\ar[rd]\ar[d, phantom, "\oplus"] & K_2\ar[d, phantom, "\oplus"] \\
\oplus & J_2\ar[d, phantom, "\oplus"] & J_2 \\
K_2\ar[rd]\ar[d, phantom, "\oplus"] & K_2\ar[uur]\ar[d, phantom, "\oplus"] & \oplus \\
J_2\ar[uur] & N_0 \oplus P_1\ar[r] & P_0 \\
\end{tikzcd}
\]
The obvious admissible monomorphisms $P_2 \to J_2$ and $P_2 \to K_2$ (the cokernel of both is $N_0$) define an admissible monomorphism from the binary acyclic complex~$T_b(\bP)$
\[\begin{tikzcd}[row sep= tiny,column sep=huge]
 & P_2\ar[r]\ar[d, phantom, "\oplus"] & P_2\ar[d, phantom, "\oplus"] \\
& P_2\ar[r]\ar[d, phantom, "\oplus"] & P_2 \\
P_2\ar[r]\ar[d, phantom, "\oplus"] & P_2\ar[d, phantom, "\oplus"] & \oplus \\
P_2\ar[r] & P_1\ar[r] & P_0 \\
\end{tikzcd}
\qquad
\begin{tikzcd}[row sep= tiny,column sep=huge]
& P_2\ar[dr]\ar[d, phantom, "\oplus"] & P_2\ar[d, phantom, "\oplus"] \\
& P_2\ar[d, phantom, "\oplus"] & P_2 \\
P_2\ar[dr]\ar[d, phantom, "\oplus"] & P_2\ar[uur]\ar[d, phantom, "\oplus"] & \oplus \\
P_2\ar[uur] & P_1\ar[r] & P_0 \\
\end{tikzcd}
\]
to (the bottom half of) $\bT'$. 
Similarly, we have an admissible epimorphism from (the top half of) $\bT'$ to the binary acyclic complex $T_f(\bM)$
\[\begin{tikzcd}[row sep= tiny,column sep=huge]
M_2\ar[r] & M_1\ar[r]\ar[d, phantom, "\oplus"] &  M_0\ar[d, phantom, "\oplus"] \\
\oplus &  M_0\ar[r]\ar[d, phantom, "\oplus"] &  M_0\\
 M_0\ar[r]\ar[d, phantom, "\oplus"] &  M_0\ar[d, phantom, "\oplus"] & \\
 M_0\ar[r] & M_0 &
\end{tikzcd}
\qquad
\begin{tikzcd}[row sep= tiny,column sep=huge]
M_2\ar[r] & M_1\ar[rd]\ar[d, phantom, "\oplus"] &  M_0\ar[d, phantom, "\oplus"] \\
\oplus &  M_0\ar[d, phantom, "\oplus"] &  M_0\\
 M_0\ar[rd]\ar[d, phantom, "\oplus"] &  M_0\ar[uur]\ar[d, phantom, "\oplus"] &  \\
 M_0\ar[uur] & M_0 &
\end{tikzcd}
\]
which obviously factors modulo $T_b(\bP)$. The kernel of the resulting epimorphism $\bT'/T_b(\bP) \rightarrow T_f(\bM)$ is
$T_{b,f}(\mathrm{sw}(\bN))$ which is obtained from $\bN$ by first switching top and bottom differential and then, similarly to $T_b(\bP)$ and $T_f(\bM)$, by adding copies of $N_2$ above $\mathrm{sw}(\bN)$ and copies of $N_0$ below $\mathrm{sw}(\bN)$. Hence we have
\[ \bT' = T_b(\bP) + T_{b,f}(\mathrm{sw}(\bN)) + T_f(\bM) \quad \textrm{in} \quad B_1(\cN).\]
Applying the switching automorphism $\tau_{P_2}$ at the appropriate place in all three degrees of $T_b(\bP)$ produces the direct sum of $\bP$ and a diagonal complex, so we obtain the relation
\begin{equation}\label{relation} T_b(\bP) = \bP + \tau_{P_2} \quad \textrm{in} \quad  L_1^2(\cN).
\end{equation}
Similarly, we obtain $T_f(\bM)=\bM+\tau_{M_0}$ and, also using \cref{lem:switching}, $T_{b,f}(\mathrm{sw}(\bN))=-\bN+\tau_{N_0}+\tau_{N_2}$.
Hence we have
\[\bT'=\bP-\bN+\bM+\tau_{P_2}+\tau_{N_0}+\tau_{N_2}+\tau_{M_0} \quad \textrm{in} \quad L_1^2(\cN).\]
Let $\bbC_i$ denote the binary acyclic complex $M_i\rightrightarrows N_i\rightrightarrows P_i$.
Filtering the total complex by columns, we similarly obtain
\[\bT'=\bbC_0-\bbC_1+\bbC_2+\tau_{M_0}+\tau_{P_1}+\tau_{M_1}+\tau_{P_2}\quad \textrm{in} \quad L_1^2(\cN).\]
The exact sequences $\bN$ and $\bbC_1$ furthermore imply the relations
\[\tau_{N_0}+\tau_{N_2}=\tau_{N_1}=\tau_{P_1}+\tau_{M_1} \quad \textrm{in} \quad B_1(\cN).\]
Hence we finally obtain the Nenashev relation
 \begin{equation}\label{NenashevRelation}\bP-\bN+\bM=\bbC_0-\bbC_1+\bbC_2 \quad \textrm{in} \quad L_1^2(\cN).
 \end{equation}
Put slightly differently, in $B_1(\cN)$, the difference of the two sides in (\ref{NenashevRelation}) is equal to the sum of (the negative of) the relation given by (\ref{relation}) and the analogous relations for $T_f(\bM)$, $T_{b,f}(\mathrm{sw}(\bN))$, $T_b(\bbC_0)$, $T_f(\bbC_2)$ and $T_{b,f}(\mathrm{sw}(\bbC_1))$. Each of these relations in turn is made up of a ladder and a diagonal relation apart from those for $T_{b,f}(\mathrm{sw}(\bN))$ and $T_{b,f}(\mathrm{sw}(\bbC_1))$ which in addition involve the ladder and diagonal relation occurring in the proof of \cref{lem:switching}.
\end{rem}

\bibliographystyle{amsalpha}
\bibliography{shortening}

\end{document}